\documentclass[reqno]{amsart}
\usepackage[english]{babel}
\usepackage{amssymb,enumerate,url}
\usepackage{palatino,mathpazo}
\usepackage[pagebackref,colorlinks]{hyperref}
\usepackage{autobreak}

\newcommand{\llangle}{{\langle\!\langle}}
\newcommand{\mathLaplace}{\Delta}

\newcommand{\rrangle}{{\rangle\!\rangle}}
\newcommand{\tmop}[1]{\ensuremath{\operatorname{#1}}}
\newcommand{\tmtextit}[1]{{\emph{#1}}}

\newenvironment{enumeratenumeric}{\begin{enumerate}[1.] }{\end{enumerate}}
\newtheorem{theorem}{Theorem}[section]
\newtheorem{corollary}[theorem]{Corollary}
\newtheorem{definition}[theorem]{Definition}
\newtheorem{lemma}[theorem]{Lemma}
\newtheorem{proposition}[theorem]{Proposition}
\newtheorem{remark}[theorem]{Remark}

\begin{document}
\begin{sloppypar}
\title[Converse to Skoda's $L^2$ Division Theorem]{A Converse to the Skoda $L^2$ Division Theorem}

\author[Z. Li]{Zhi Li}
\address{Zhi Li: School of Mathematical Sciences and Key Laboratory of Mathematics and Information Networks (Ministry of Education), Beijing University of Posts and Telecommunications, Beijing 100876, China}
\email{lizhi@amss.ac.cn, lizhi10@foxmail.com}

\author[X. Meng]{Xiankui Meng}
\address{Xiankui Meng: School of Mathematical Sciences and Key Laboratory of Mathematics and Information Networks (Ministry of Education), Beijing University of Posts and Telecommunications, Beijing 100876, China}
\email{mengxiankui@amss.ac.cn}

\author[J. Ning]{Jiafu Ning}
\address{Jiafu Ning: School of Mathematics and Statistics, HNP-LAMA, Central South University, Changsha, Hunan 410083, China}
\email{jfning@csu.edu.cn}

\author[X. Zhou]{Xiangyu Zhou}
\address{Xiangyu Zhou: Institute of Mathematics\\Academy of Mathematics and Systems Sciences\\and Hua Loo-Keng Key
	Laboratory of Mathematics\\Chinese Academy of
	Sciences\\Beijing\\100190\\China}
\email{xyzhou@math.ac.cn}


\begin{abstract}
 In this paper, we present a converse to a version of Skoda's $L^2$ division theorem by investigating the solvability of $\bar{\partial}$ equations of a specific type.
\end{abstract}

{\maketitle}

\section{Introduction}

The present paper is devoted to a converse to the Skoda $L^2$ division
theorem. Consider a pseudoconvex domain $D \subset \mathbb{C}^n$ and let $f$
be a holomorphic function on $D$. Let $g_1, g_2, \ldots, g_{k}$ be $k$ holomorphic functions with no common zeros in $D$. The classical division problem
is to find holomorphic functions $h_1, h_2, \ldots, h_{k}$ such that
\[ f  = g_1 h_1 + \cdots + g_{k} h_{k} . \]
To determine the condition under which a solution exists for the division
problem is an important question in several complex variables.
It is obtained by Oka {\cite{oka1950}} that for any holomorphic function $f$
which is locally in the ideal generated by $g_j$ $(1 \leqslant j \leqslant k)$, the classical division problem is always solvable on pseudoconvex domains.

It is natural to consider the division problem with $L^2$
estimates. Skoda {\cite{skoda1972}}
established the following well-known result, now referred to as the Skoda $L^2$
division theorem, which can be regarded as the $L^2$ version of Oka's division
theorem.

\begin{theorem}[Skoda {\cite{skoda1972,skoda1978}}]
  \label{thm:skoda-l2-division} Let $D \subset \mathbb{C}^n$ be a pseudoconvex
  domain. If $\varphi$ is a plurisubharmonic function on $D$, then for any $g = (g_1,
  \ldots, g_{k}) \in (\mathcal{O} (D))^{\oplus k}$ whose components have no common zeros and $\varepsilon > 0$, and any $f \in \mathcal{O} (D)$ satisfying
  \[ \int_D \frac{| f |^2}{| g |^2} e^{- \varphi - p (1 + \varepsilon) \log |
     g |^2} d V < \infty, \]
  there exists $h = (h_1, \ldots, h_{k}) \in (\mathcal{O} (D))^{\oplus k}$ such that $f =  g_1 h_1 + \cdots + g_{k} h_{k}$
  and
  \[ \int_D | h |^2 e^{- \varphi - p (1 + \varepsilon) \log | g |^2} d V
     \leqslant \left( 1 + \frac{1}{\varepsilon} \right) \int_D \frac{| f
     |^2}{| g |^2} e^{- \varphi - p (1 + \varepsilon) \log | g |^2} d V, \]
  where $| g |^2 = | g_1 |^2 + \cdots + | g_{k} |^2$, $| h |^2 = | h_1 |^2
  + \cdots + | h_{k} |^2$, $p = \min \{ n, k-1 \}$ and $d V$ is the Lebesgue
  measure on $\mathbb{C}^n$.
\end{theorem}

The Skoda $L^2$ division theorem is important in several complex variables,
complex analytic and algebraic geometry. It is the main ingredient of the
Brian{\c c}on-Skoda theorem {\cite{briancon-skoda1974}}. With Skoda's result,
Siu established the deformation invariance of plurigenera and the finite
generation of the canonical ring of complex projective manifolds of
general type (see {\cite{siu1998,siu2002,siu2004,siu2005,siu2009tech}}). These
results have had a profound impact on complex algebraic geometry. The Skoda
$L^2$ division theorem can be viewed as an effective analogue of Hilbert's
Nullstellensatz. Brownawell {\cite{brownawell1987}}, Ein-Lazarsfeld
{\cite{ein-lazarsfeld1999}} have provided effective versions of the
Nullstellensatz using Skoda's result.

The original proof of Theorem \ref{thm:skoda-l2-division} employed a
combination of functional analysis techniques and a specialized linear
algebraic approach. Subsequently, Skoda {\cite{skoda1978}} extended Theorem
\ref{thm:skoda-l2-division} to the $L^2$ surjectivity of bundle morphisms and
deduced the $L^2$ surjectivity to a $\bar{\partial}$-equation of the kernel
bundle. In fact, in Skoda's original paper {\cite{skoda1978}} (see also
{\cite{demailly-smallbook}}), a stronger version of Skoda's $L^2$ theorem was
obtained implicitly.  A similar
$L^2$ division result can also be found in {\cite{ohsawa-book2002}}.

In order to state the stronger version of Skoda's result,
let us recall some notation. Let $\underline{\mathbb{C}}^{k} \rightarrow D$ and $\underline{\mathbb{C}}
\rightarrow D$ denote the trivial vector bundle of rank $k$ over $D$ and
the trivial line bundle over $D$ respectively. Then $g = (g_1, g_2, \ldots,
g_{k})$ induces a morphism from $\underline{\mathbb{C}}^{k}$ to
$\underline{\mathbb{C}}$ by sending a holomorphic section $h$ of
$\underline{\mathbb{C}}^{k}$ to $g \cdot h:= g_1 h_1 + \cdots + g_{k} h_{k}$. For simplicity, we will
continue to denote this morphism by $g$ when there is no confusion.  If the components of $g=(g_1,\ldots,g_{k})$
have no common zeros, then we have the following exact sequence:
\[
0\rightarrow S \rightarrow \underline{\mathbb{C}}^{k}\xrightarrow{g} \underline{\mathbb{C}}\rightarrow 0,
\]
where $S=\operatorname{Ker}g$. Denote by $\beta=\left|g|^2(-\partial\left({\frac{g_1}{|g|^2}}\right),\ldots,-\partial\left({\frac{g_{k}}{|g|^2}}\right)\right)$
the second fundamental form of $S$ (cf. Section \ref{sec:prelimaries}).
Given a
holomorphic section $f$ of $\underline{\mathbb{C}}$, the division problem is
equivalent to finding a holomorphic section $h$ of $\underline{\mathbb{C}}^{k}$ such that $g \cdot h = f$ (with $L^2$ control).

As usual, we write $d z=d z_1 \wedge \cdots \wedge d z_{n}$ and $\| u \|_{\varphi}^2=\int_D |u|^2 e^{-\varphi} d V$ for $u\in (L^2(D))^{\oplus k}$. Denote by $\Lambda$ the adjoint operator of $L=\omega\wedge\cdot$ and $\omega$ the standard Euclidean metric on $\mathbb{C}^n$. We denote
\[
I_{\varphi,\psi, g}(f)=\int_D \left\{\frac{1}{| g |^2}  +
     \left( \left[ i \partial \bar{\partial} \psi \otimes
    \tmop{Id}_{\underline{\mathbb{C}}^{k}} + i \beta^{\ast} \wedge \beta,
    \Lambda \right]^{- 1} \beta^{\ast} \wedge d z, \beta^{\ast} \wedge d z \right)\right\}|f|^2 e^{- \varphi
    - \psi} d V.
\]
Before we state the stronger version of Skoda's $L^2$ division theorem (which implies Theorem \ref{thm:skoda-l2-division}), we introduce the following notion:
\begin{definition}
Let $D\subset\mathbb{C}^n$ be a domain and $\varphi\in L_{\operatorname{loc}}^1(D)$. We call a holomorphic function $f\in \mathcal{O}(D)$ is $L^2$ divisible with respect to $\varphi$ if for any $k\in \mathbb{N}$ and any $g = (g_1,
  \ldots, g_{k}) \in (\mathcal{O} (D))^{\oplus k}$ whose components have no common zeros and any smooth plurisubharmonic function $\psi$ on $D$ satisfying the following conditions:

  \begin{enumerate}
  \item $g_1,  g_2, \ldots, g_{k}$ have no common zeros in $D$;
  \item $i \partial \bar{\partial} \psi \otimes
     \tmop{Id}_{\underline{\mathbb{C}}^{k}} + i \beta^{\ast} \wedge \beta
     > 0 $ in the sense of Nakano;
  \item $I_{\varphi,\psi,g}(f)<\infty$,
\end{enumerate}
there exists $h \in (\mathcal{O} (D))^{\oplus k}$ such that $f = g
  \cdot h$ and
  \[
  \|h\|^2_{\varphi+\psi}\leqslant I_{\varphi,\psi,g}(f).
  \]
\end{definition}
With the above definition, the stronger version of Theorem \ref{thm:skoda-l2-division} obtained by Skoda \cite{skoda1978} can be formulated as follows.
\begin{theorem}[Skoda {\cite{skoda1978}}]
  \label{thm:sharper-division} Let $D\subset\mathbb{C}^n$ be a pseudoconvex domain. If $\varphi$ is plurisubharmonic on $D$, then any $f\in\mathcal{O}(D)$ is $L^2$ divisible with respect to $\varphi$.
\end{theorem}

Recently, motivated by the geometric meaning of the optimal $L^2$ extension theorem found in \cite{guan-zhou2014optimal} (see also \cite{zhou2015}, \cite{Guan-Zhou2017}), new methods
were found in
{\cite{deng-wang-zhang-zhou2018new-char,hosono-inayama,deng-ning-wang-zhou2023}} to characterize plurisubharmonic functions and
Nakano positivity in terms of $L^2$ extensions and $L^2$ estimates, which can
be seen as a converse to the $L^2$ theory in several complex variables. Related work was done in \cite{LZ}. It is thus natural to ask whether a converse to the
Skoda(-type) division theorem exists.

As usual, we denote by $A^2(D,\varphi)$ the set of holomorphic functions on $D$ which are square
integrable with respect to $e^{-\varphi}$. We will always assume that $A^2(D,\varphi)\neq\{0\}$ in the rest of the paper. The main result of the present paper, stated below, provides a converse to Theorem \ref{thm:sharper-division}.
\begin{theorem}
  \label{thm:2-division}
  Let $D\subset\mathbb{C}^n$ be a bounded domain and $\varphi\in C^2(D)$. If there exists some non-zero $f\in A^2(D,\varphi)$ which is $L^2$ divisible with respect to $\varphi$, then $\varphi$ is plurisubharmonic on $D$.
   \end{theorem}

If $\varphi$ is smooth up to the boundary of $D$, we can get the plurisubharmonicity of $\varphi$ by investigating the $L^2$ divisibility of the constant function $1$. Theorem \ref{thm:2-division} will follow from the following result.

\begin{theorem}
  \label{thm:main-inverse} Let $D \subset \mathbb{C}^n$ be a bounded domain,
  $\varphi \in C^2 (D)$. Suppose $r$ is an integer between $1$ and $n$. If there exists a non-zero holomorphic function $f\in A^2(D,\varphi)$ such that for any
   $g = (g_1, \ldots, g_{r + 1}) \in (\mathcal{O} (D))^{\oplus (r +
  1)}$  whose components have no common zeros and
  $d g_1\wedge\cdots\wedge d g_r \neq 0$, and any smooth $\psi\in PSH(D) $ satisfying
  \[ i \partial \bar{\partial} \psi \otimes
     \tmop{Id}_{\underline{\mathbb{C}}^{r + 1}} + i \beta^{\ast} \wedge \beta
     > 0 \]
  in the sense of Nakano and $I_{\varphi,\psi,g}(f)<\infty$,
  there exists $h = (h_1, \ldots, h_{r + 1}) \in (\mathcal{O} (D))^{\oplus (r
  + 1)}$ such that $f = g \cdot h$ and
  \[
  \|h\|^2_{\varphi+\psi}\leqslant I_{\varphi,\psi,g}(f),
  \]
  then the sum of any $r$ eigenvalues of $i
  \partial \bar{\partial} \varphi$ is non-negative.
\end{theorem}

    As an application of Theorem \ref{thm:2-division}, we can give a new proof of complex Pr\'ekopa's minimum principle (Theorem \ref{thm:application-complex-prekopa})  obtained by Berndtsson \cite{berndtsson1998}. Another immediate consequence of Theorem \ref{thm:2-division} is the following new characterization of plurisubharmonic functions.
\begin{corollary}
Let $\Omega\subset\mathbb C^n$ be a domain and let
$\varphi\in C^2(\Omega)$.  The following statements are equivalent.
\begin{enumerate}[(i)]
\item The function $\varphi$ is plurisubharmonic on $\Omega$.
\item For every $x\in\Omega$, there exists a Euclidean ball
      $B_x\Subset\Omega$ centered at $x$ such that the constant function
      $1$ is $L^2$ divisible on $B_x$ with respect to
      $\varphi|_{B_x}$.
\item For every $x\in\Omega$, there exists a bounded pseudoconvex
      neighborhood $U_x\Subset\Omega$ and a non-zero function
      $f_x\in A^2(U_x,\varphi)$ which is $L^2$ divisible on $U_x$ with
      respect to $\varphi|_{U_x}$.
\end{enumerate}
\end{corollary}

The present paper is structured as follows. In Section
\ref{sec:prelimaries}, we provide a review of fundamental concepts related to
hermitian holomorphic vector bundles, including Nakano
positivity, the Bochner-Kodaira-Nakano identity, and the Gauss-Codazzi formula. In
Section \ref{sec:sharper-division}, we recall the proof of the stronger version of Skoda's
$L^2$ division theorem (Theorem \ref{thm:sharper-division}). In Section \ref{sec:inverse-division}, we present
the proof of Theorem
\ref{thm:main-inverse},  as well as a new proof of complex Pr\'ekopa's minimum principle obtained by Berndtsson \cite{berndtsson1998}.

\section{Preliminaries}\label{sec:prelimaries}

\subsection{Foundations for hermitian holomorphic vector bundles}

We recall some notions related to the positivity of hermitian holomorphic
vector bundles as well as the well-known Bochner-Kodaira-Nakano identity.

Let $(X, \omega)$ be a K{\"a}hler manifold of dimension $n$ and $(E, h)
\rightarrow X$ a hermitian holomorphic vector bundle of rank $r$. For any $u,
v \in E_x \otimes \wedge^{p, q} T^{\ast}_x X$, we denote the inner product of
$u, v$ by
$\langle u, v \rangle$.
Similarly, we denote $| u |^2 = \langle u, u \rangle$,
\[ \| u \|^2 = \int_X | u |^2 d V_{\omega} \quad \text{and}\quad  \llangle u, v \rrangle =
   \int_X \langle u, v \rangle d V_{\omega} . \]

Denote the Chern connection of $(E, h)$ by $D_E=D'_E+\bar{\partial}_E$.
The curvature of $(E, h)$
is denoted by $\Theta_{E, h}$, and when there is no need to specify the metric
$h$, we simply write $\Theta (E)$. It can be expressed in terms of an
orthonormal frame $(e_1, \ldots, e_r)$ of $E$ over a coordinate neighborhood
$(U, z_1, \ldots, z_n)$ as follows:
\[ i \Theta_{E, h} = i \sum_{1 \leqslant j, k \leqslant n, 1 \leqslant
   \lambda, \mu \leqslant r} c_{j k \lambda \mu} d z_j \wedge d \bar{z}_k
   \otimes e^{\ast}_{\lambda} \otimes e_{\mu}, \]
with $\bar{c}_{j k \lambda \mu} = c_{k j \mu \lambda}$. We associate with $i
\Theta_{E, h}$ a Hermitian form on $T X \otimes E$ as follows:
\[ \theta_{E, h} (u, v) = \sum_{j, k, \lambda, \mu} c_{j k \lambda \mu} u_{j
   \lambda} \bar{v}_{k \mu}, \]
for $u, v \in T_x X \otimes E_x$. $(E, h)$ is called (semi-)positive in the sense of Nakano if for any
    non-zero $u$, one has $\theta_{E, h} (u, u) >0$ ($\theta_{E, h} (u, u) \geqslant 0$). For $B\in T^\ast\otimes \operatorname{Hom}(E,E)$, we also write $B>0$ in the rest of the paper if the corresponding quadratic form is positive definite.

The formal adjoints of $D'_E$ and $\bar{\partial}_E$ are denoted by ${D'_E}^{\ast}$ and $\bar{\partial}_E^{\ast}$ respectively.
The K{\"a}hler identities provide explicit relations between the operators
$D_E'$, $\bar{\partial}_E $ and their formal adjoints ${D'_E}^{\ast}$ and
$\bar{\partial}_E^{\ast}$.

\begin{proposition}[See {\cite{demailly-bigbook}}]
  \label{prop:kahler-identities}Let $(X, \omega)$ be a K{\"a}hler manifold and
  $(E, h) \rightarrow X$ a hermitian holomorphic vector bundle. Let $L=\omega\wedge\cdot$
  and let $\Lambda$ be the formal adjoint of $L$. Then one has
  \begin{enumeratenumeric}
    \item $[\bar{\partial}_E^{\ast}, L] = i D'_E$,

    \item $[{D'_E}^{\ast}, L] = - i \bar{\partial}_E$,

    \item $[\Lambda, \bar{\partial}_E] = - i {D'_E}^{\ast}$,

    \item $[\Lambda, D'_E] = i \bar{\partial}^{\ast}_E$.
  \end{enumeratenumeric}
\end{proposition}

The Bochner-Kodaira-Nakano identity plays
a crucial role in the $L^2$ estimates for the $\bar{\partial}$-operator. Let
\[ \mathLaplace'_E = D'_E {D'}^{\ast}_E {+ D'}^{\ast}_E D'_E, \quad \Delta''_E =
   \bar{\partial}_E \bar{\partial}^{\ast}_E + \bar{\partial}_E^{\ast}
   \bar{\partial}_E, \]
be the $D'_E$-Laplacian and $\bar{\partial}_E$-Laplacian respectively.

\begin{proposition}[Bochner-Kodaira-Nakano, see {\cite{demailly-bigbook}}]
  \label{prop:bochner-kodaira}Let $(X, \omega)$ be a K{\"a}hler manifold and $(E,
  h) \rightarrow X$ a hermitian holomorphic vector bundle. Then one has
  \[ \Delta''_E = \Delta_E' + [i \Theta (E), \Lambda] . \]
\end{proposition}

Fix $x_0 \in X$ and choose local coordinates $(z_1, \ldots, z_n)$ centered at
$x_0$ such that $\left( \frac{\partial}{\partial z_1}, \ldots,
\frac{\partial}{\partial z_n} \right)$ forms an orthonormal basis of $T_{x_0}
X$. Then the K{\"a}hler form $\omega$ can be expressed as
\[ \omega = i \sum_{j=1}^n d z_j \wedge d \bar{z}_j + O (\| z \|^2). \]
Moreover, at $x_0$, we also have
\[ i \Theta (E)  = i \sum c_{j k \lambda \mu} d z_j \wedge d \bar{z}_k
   \otimes e^{\ast}_{\lambda} \otimes e_{\mu}, \]
where $(e_1, \ldots, e_r)$ is an orthonormal basis of $E_{x_0}$. If $u = \sum
u_{j \lambda} d z \wedge d \bar{z}_j \otimes e_{\lambda} \in \wedge^{n, 1}
T^{\ast} X \otimes E$, where $d z = d z_1 \wedge \cdots \wedge d z_n$, then
\begin{equation}
  \langle [i \Theta (E), \Lambda] u, u \rangle = \sum_{j, k, \lambda, \mu}
  c_{j k \lambda \mu} u_{j \lambda} \bar{u}_{k \mu} .
  \label{equ:curvature-operator}
\end{equation}
\begin{lemma}[See also {\cite{demailly-bigbook}}]
  \label{lem:positive-definite}Let $(X, \omega)$ be a K{\"a}hler manifold and
  $(E, h) \rightarrow X$ a hermitian holomorphic vector bundle. Then $(E, h)$
  is (semi-)positive in the sense of Nakano if and only if $[i \Theta (E),
  \Lambda]$ is (semi-)positive definite on $\wedge^{n, 1} T^{\ast} X \otimes
  E$.
\end{lemma}

\subsection{The Gauss-Codazzi formula}

In this subsection, we recall the Gauss-Codazzi formula. The reader may refer to
{\cite{demailly-bigbook,kobayashi-diff-geo-complex-vs}} for details. Let $(X,
\omega)$ be a K{\"a}hler manifold of dimension $n$ and consider the following
exact sequence of holomorphic vector bundles over $X$,
\[ 0 \rightarrow S \xrightarrow{j} E \xrightarrow{g} Q \rightarrow 0. \]
Assume that $E$ admits a hermitian metric, and equip $S$ and $Q$ with the
induced metrics. Denote by $D_E$, $D_S$, $D_Q$ the Chern connections
corresponding to $E, S, Q$ respectively. Write $D'_E$ and $\bar{\partial}_E$ for
the $(1, 0)$ and $(0, 1)$ parts of $D_E$ and similarly for
$D'_S$, $\bar{\partial}_S$ and $D'_Q$, $\bar{\partial}_Q$.

Note that there always exist smooth homomorphisms
\[ j^{\ast} : E \rightarrow S, \quad g^{\ast} : Q \rightarrow E \]
such that
\[ j^{\ast} \circ j = \tmop{id}_S, \quad g \circ g^{\ast} = \tmop{id}_Q . \]
Thus they yield a $C^{\infty}$ isomorphism
\[ j^{\ast} \oplus g : E \simeq S \oplus Q. \]
\begin{proposition}
  \label{prop:gc-connection}According to the $C^{\infty}$ isomorphism
  $j^{\ast} \oplus g$, $D_E$ can be written as
  \[ D_E = \left[\begin{array}{cc}
       D_S & - \beta^{\ast} \\
       \beta & D_Q
     \end{array}\right], \]
  where $\beta$ is a smooth $(1, 0)$ form valued in $\tmop{Hom} (S, Q)$.
  $\beta$ is called the second fundamental form of $S$ in $E$, and $\beta^{\ast}$
  is the adjoint of $\beta$. Moreover, one also has
  \[ \bar{\partial} g^{\ast} = -  \beta^{\ast} . \]
\end{proposition}

It follows from Proposition \ref{prop:gc-connection} that if $u$ is a smooth
section of $S$, $D_S u = D_E u - \beta u$. Actually, it can be verified that
$\beta u$ is the orthogonal projection of $D_E u$ onto $Q$.
Since $\beta$ is a $\tmop{Hom} (S, Q)$-valued $(1, 0)$ form, by decomposing
$D_E$ and $D_S$ into $(1, 0)$ and $(0, 1)$ parts, one has
\[ D'_S = D'_E - \beta, \quad \bar{\partial}_S = \bar{\partial}_E . \]
Consequently, according to Proposition \ref{prop:kahler-identities}, we have:
\begin{equation}\label{equ:chern-adjoint}
 {i D'}^{\ast}_S = - [\Lambda, \bar{\partial}_S] = - [\Lambda,
   \bar{\partial}_E] {= i D'}^{\ast}_E .
\end{equation}
This implies that the formal adjoint operators of $D'_E$ and $D'_S$ coincide,
despite their difference in the second fundamental form $\beta$.

Write $\Theta (E)$ for the curvature of $E$ and similarly for $\Theta (S)$ and $\Theta
(Q)$. The following proposition gives the relation between the curvatures of
$S, E, Q$.

\begin{proposition}
  \label{prop:gc-curvature}With the same notation as in Proposition
  \ref{prop:gc-connection}, we have $\bar{\partial} \beta^{\ast} = 0$ and
  \[ \Theta (E) = \left[\begin{array}{cc}
       \Theta (S) - \beta^{\ast} \wedge \beta & -D' \beta^{\ast} \\
       \bar{\partial} \beta & \Theta (Q) - \beta \wedge \beta^{\ast}
     \end{array}\right] . \]
\end{proposition}

It follows from Proposition \ref{prop:gc-curvature} that if $u$ is a smooth
section of $S$, then
\[
\Theta (S) u = \Theta (E) u + \beta^{\ast} \wedge \beta u.
\]
If $\beta$ is written as $\sum d z_j \otimes \beta_j$, $\beta_j \in \tmop{Hom}
(S, Q)$, then
\[ \beta^{\ast} \wedge \beta = - \sum d z_j \wedge d \bar{z}_k \otimes
   \beta^{\ast}_k \beta_j \]
and
\[ \langle [i \Theta (S), \Lambda] u, u \rangle_S = \langle [i \Theta (E),
   \Lambda] u, u \rangle_E - | \beta u |^2_Q. \]
Therefore, the curvature of $S$ differs from that of $E$ by a Nakano negative
part.

\subsection{$L^2$ estimates}

In this subsection, we recall Demailly's $L^2$ existence theorem for solving the $\bar{\partial}$ equation
on K{\"a}hler manifolds.

\begin{theorem}[See {\cite{demailly1982}}]
  \label{thm:l2-estimates}Let $X$ be a complete K{\"a}hler manifold equipped
  with a K{\"a}hler metric $\omega$ which is not necessarily complete. Let
  $(E, h)$ be a hermitian holomorphic vector bundle of rank $r$ over $X$ and
  assume that the curvature operator $[i \Theta(E), \Lambda]$ is
  semi-positive definite everywhere on $\wedge^{n, q} T_X^{\ast} \otimes E$
  for some $q \geqslant 1$. Then for any $g \in L^2_{\tmop{loc}} (X,
  \wedge^{n, q} T_X^{\ast} \otimes E)$ such that $\bar{\partial} g = 0$ and
  \[ \int_X \langle [i \Theta(E), \Lambda]^{- 1} g, g \rangle d V_{\omega} < \infty, \]
  there exists $f \in L^2 (X, \wedge^{n, q-1} T_X^{\ast} \otimes E)$ such that
  $\bar{\partial} f = g$ and
  \[ \int_X | f |^2 d V_{\omega} \leqslant \int_X \langle [i
     \Theta(E), \Lambda]^{- 1} g, g \rangle d
     V_{\omega}. \]
\end{theorem}

\section{The Skoda $L^2$ Division Theorem}\label{sec:sharper-division}

\subsection{The stronger version of Skoda's $L^2$ division theorem}

The stronger version of Skoda's
$L^2$ division theorem (Theorem \ref{thm:sharper-division}) is
essentially contained in {\cite{skoda1978}} (see also {\cite{demailly-smallbook}})
 as we have mentioned previously. For the
sake of completeness, we give its proof here.

\begin{proof}
  Since $D$ is pseudoconvex, we may assume $\psi$ is smooth. Note that
  \[
    g^{\ast} f = \left( \frac{\bar{g}_1}{| g |^2}, \ldots, \frac{\bar{g}_{r +
    1}}{| g |^2} \right)f 
  \] 
  is a smooth lifting of $f$ to $\underline{\mathbb{C}}^{r + 1}$.
  Since any section $h$ of $\underline{\mathbb{C}}^{r + 1}$ such that $g \cdot
  h = f$ differs from $g^{\ast} f$ by a section of $S = \tmop{Ker} g$, we want to
  find a section $u$ of $S$ such that
  \[ h = g^{\ast} f + u \]
satisfying
  \begin{equation}
    \bar{\partial} h = \bar{\partial} (g^{\ast} f) + \bar{\partial} u =
    0. \label{equ:proof-division-dbar}
  \end{equation}
  It follows from the Gauss-Codazzi formula that \eqref{equ:proof-division-dbar}
  can be reformulated as
  \[ \bar{\partial} u =-\bar{\partial} (g^{\ast} )f= \beta^{\ast} f. \]
  If we equip $\underline{\mathbb{C}}^{r
  + 1}$ with the metric $\tmop{Id}_{\underline{\mathbb{C}}^{r + 1}} \otimes
  e^{- \psi}$, by the Gauss-Codazzi formula, the curvature operator of $S$ is
  \[ \left[ i \partial \bar{\partial} \psi \otimes
     \tmop{Id}_{\underline{\mathbb{C}}^{r + 1}} + i \beta^{\ast} \wedge \beta,
     \Lambda \right] . \]
  It then follows from our hypothesis and Theorem \ref{thm:l2-estimates} that
  there exists an $S$ valued section $u $ such that
  \[ \bar{\partial} u = \beta^{\ast} f \]
  and
  \[ \|u\|^2_{\varphi+\psi} \leqslant \int_D \left( \left[ i
     \partial \bar{\partial} \psi \otimes \tmop{Id}_{\underline{\mathbb{C}}^{r
     + 1}} + i \beta^{\ast} \wedge \beta, \Lambda \right]^{- 1} \beta^{\ast}
     \wedge d z, \beta^{\ast} \wedge d z \right)|f|^2 e^{- \varphi - \psi} d V. \]
  Since $u$ is $S$-valued, we conclude that
  \[
  \|h\|^2_{\varphi+\psi}
     =  \int_D | g^{\ast} f + u |^2 e^{- \varphi - \psi} d V\leqslant  I_{\varphi,\psi,g}(f).
  \]
\end{proof}

  If one chooses $\psi = p (1 + \varepsilon) \log | g
  |^2$, where $p = \min \{ n, r \}$, then it follows from Skoda's Lemme
  Fondamental {\cite{skoda1972}} that
  \[ i \partial \bar{\partial} \psi \otimes
     \tmop{Id}_{\underline{\mathbb{C}}^{r + 1}} \geqslant - (1 + \varepsilon)
     i \beta^{\ast} \wedge \beta. \]
  It then follows that
  \begin{eqnarray*}
    &  & \int_D \left( \left[ i \partial \bar{\partial} \psi \otimes
    \tmop{Id}_{\underline{\mathbb{C}}^{r + 1}} + i \beta^{\ast} \wedge \beta,
    \Lambda \right]^{- 1} \beta^{\ast} \wedge d z, \beta^{\ast} \wedge d z \right) |f|^2 e^{- \varphi
    - p (1 + \varepsilon) \log | g |^2} d V\\
    & \leqslant & \int_D ([- \varepsilon i \beta^{\ast} \wedge \beta,
    \Lambda]^{- 1} \beta^{\ast} \wedge d z, \beta^{\ast} \wedge d z) |f|^2 e^{- \varphi - p (1 +
    \varepsilon) \log | g |^2} d V\\
    & \leqslant & \frac{1}{\varepsilon} \int_D \frac{| f |^2}{| g |^2} e^{-
    \varphi - p (1 + \varepsilon) \log | g |^2} d V.
  \end{eqnarray*}
  Therefore, we have
  \begin{eqnarray*}
    &  & \int_D | h |^2 e^{- \varphi - p (1 + \varepsilon) \log | g |^2} d
    V\\
    & = & \int_D \frac{| f |^2}{| g |^2} e^{- \varphi - \psi} d V + \int_D
    \left( \left[ i \partial \bar{\partial} \psi \otimes
    \tmop{Id}_{\underline{\mathbb{C}}^{r + 1}} + i \beta^{\ast} \wedge \beta,
    \Lambda \right]^{- 1} \beta^{\ast} \wedge d z, \beta^{\ast} \wedge d z \right)|f|^2 e^{- \varphi
    - \psi} d V\\
    & \leqslant & \int_D \frac{| f |^2}{| g |^2} e^{- \varphi - p (1 +
    \varepsilon) \log | g |^2} d V + \frac{1}{\varepsilon} \int_D \frac{| f
    |^2}{| g |^2} e^{- \varphi - p (1 + \varepsilon) \log | g |^2} d V\\
    & = & \left( 1 + \frac{1}{\varepsilon} \right) \int_D \frac{| f |^2}{| g
    |^2} e^{- \varphi - p (1 + \varepsilon) \log | g |^2} d V.
  \end{eqnarray*}
  Thus we recover the original Skoda $L^2$ division theorem.

  If  $\varphi = 0$ and $\psi$ satisfies
  \[ i \partial \bar{\partial} \psi \otimes
     \tmop{Id}_{\underline{\mathbb{C}}^{r + 1}} \geqslant - i \beta^{\ast} \wedge \beta +
     \varepsilon \omega, \]
  where $\omega$ is the standard Euclidean metric on $\mathbb{C}^n$ and
  $\varepsilon > 0$, then it follows that
  \[ \int_D \left( \left[ i \partial \bar{\partial} \psi \otimes
     \tmop{Id}_{\underline{\mathbb{C}}^{r + 1}} + i \beta^{\ast} \wedge \beta,
     \Lambda \right]^{- 1} \beta^{\ast} \wedge d z, \beta^{\ast} \wedge d z \right)|f|^2 e^{- \psi} d
     V \leqslant \frac{1}{\varepsilon} \int_D | \beta^{\ast}|^2| f |^2 e^{- \psi}
     d V. \]
  So we have the $L^2$ division theorem in {\cite{ohsawa-book2002}}.

\begin{theorem}[See {\cite{ohsawa-book2002}}]
  Let $D$ be a pseudoconvex domain. Assume that $g = (g_1, \ldots, g_{r + 1})
  \in (\mathcal{O} (D))^{\oplus (r + 1)}$ whose components have
  no common zeros on $D$ and let $\psi$ be a $C^{\infty}$ plurisubharmonic
  function such that
  \[ i \partial \bar{\partial} \psi \otimes
     \tmop{Id}_{\underline{\mathbb{C}}^{r + 1}} \geqslant - i \beta^{\ast} \wedge \beta +
     \varepsilon \omega, \]
  for some $\varepsilon > 0$. Then for any holomorphic function $f$ such that
  \[ \int_D\left( \frac{1}{| g |^2}
      + \frac{1}{\varepsilon}  | \beta^{\ast}
     |^2\right)| f |^2 e^{- \psi} d V <
     \infty, \]
  there exists $h = (h_1, \ldots, h_{r + 1}) \in (\mathcal{O} (D))^{\oplus (r
  + 1)}$ such that $f = g \cdot h$ and
  \[ \|h\|^2_{\psi} \leqslant \int_D\left( \frac{1}{| g |^2}
      + \frac{1}{\varepsilon}  | \beta^{\ast}
     |^2\right)| f |^2 e^{- \psi} d V. \]
\end{theorem}

\subsection{$L^2$ divisions and $\bar{\partial}$-equations}

The $L^2$ division problem is equivalent to certain $L^2$ estimates for $\bar{\partial}$-equations with values in the kernel bundle $S$.

\begin{proposition}
  \label{prop:dbar-division-equivalent} Let $D \subset \mathbb{C}^n$ be a
  domain. Let $g=(g_1, g_2, \ldots, g_{r + 1}) \in (\mathcal{O} (D))^{\oplus (r + 1)}$ be such that its components have no common zeros. Let $\varphi \in L^1_{\tmop{loc}} (D)$ and $\psi \in C^{\infty}(D)$. Suppose that
  \[ \left[ i \partial \bar{\partial} \psi \otimes
     \tmop{Id}_{\underline{\mathbb{C}}^{r + 1}} + i \beta^{\ast} \wedge \beta,
     \Lambda \right] \]
   is positive definite. If $f \in
    \mathcal{O} (D)$ satisfies $I_{\varphi,\psi,g}(f)<\infty$,
  then the following statements are equivalent:
  \begin{enumerate}[a)]
    \item
    There exists $h \in (\mathcal{O} (D))^{\oplus (r + 1)}$ such that $f = g
    \cdot h$ and
    \begin{equation}\label{equ:sharper-skoda-estimate}
    \|h\|^2_{\varphi+\psi}\leqslant I_{\varphi,\psi,g}(f).
    \end{equation}
    \item
    There exists an $S$-valued section $u$ such that
    $
      \bar{\partial} u = \beta^{\ast} f
    $
    and
    \begin{equation}
      \label{equ:kernel-bundle-estimates}
      \|u\|^2_{\varphi+\psi} \leqslant I_{\varphi,\psi,g}(f)-\int_D \frac{| f |^2}{| g |^2} e^{- \varphi - \psi} d V.
    \end{equation}
  \end{enumerate}
\end{proposition}

\begin{proof}
    $b) \Rightarrow a)$. It is obvious from the proof of Theorem   \ref{thm:sharper-division}.

  $a) \Rightarrow b)$.
  \[ \frac{f}{| g |^2} = \frac{g_1}{| g |^2} h_1 + \cdots + \frac{g_{r + 1}}{|
     g |^2} h_{r + 1} . \]
  It is easy to check that
  \begin{equation}
    g^{\ast} f = \left( \frac{\bar{g}_1}{| g |^2}, \ldots, \frac{\bar{g}_{r +
    1}}{| g |^2} \right)f \label{equ:division-equation1}
  \end{equation}
  satisfies $g\cdot g^{\ast}f=f$. However, $g^{\ast} f$ fails to be
  holomorphic. It is obvious that the $S$-valued section $u = h - g^{\ast} f$ satisfies
  the $\bar{\partial}$-equation
  \[ \bar{\partial} u = - \bar{\partial} (g^{\ast} f) = \beta^{\ast} f. \]
  It then follows from $\bar{\partial} u = \beta^{\ast} f$ that
    \begin{eqnarray*}
    &  & I_{\varphi,\psi,g}(f)\\
    & \geqslant & \int_D | h |^2 e^{- \varphi - \psi} d V\\
    & = & \int_D | u + g^{\ast} f |^2 e^{- \varphi - \psi} d V\\
    & = & \int_D | u |^2 e^{- \varphi - \psi} d V + \int_D | g^{\ast} f |^2
    e^{- \varphi - \psi} d V\\
    & = & \int_D | u |^2 e^{- \varphi - \psi} d V + \int_D \frac{| f |^2}{| g
    |^2} e^{- \varphi - \psi} d V,
  \end{eqnarray*}
  where the third equality holds because $g^{\ast} f \bot S$. Therefore,
    \[ \|u\|^2_{\varphi+\psi} \leqslant I_{\varphi,\psi,g}(f)-\int_D \frac{| f |^2}{| g |^2} e^{- \varphi - \psi} d V. \]
\end{proof}

    By Proposition \ref{prop:dbar-division-equivalent}, to prove Theorem
    \ref{thm:main-inverse} it suffices to consider
   $\bar{\partial}$-equations of the special form $\bar{\partial} u = \beta^{\ast} f$.
    The nonhomogeneous term in this $\bar{\partial}$-equation has a fixed form and does not have compact support. This fact leads us to a global problem. Moreover, by employing the Gauss-Codazzi formula, the
    curvature operator of $S$ naturally possesses a Nakano
    negative part $i \beta^{\ast} \wedge \beta$. These are the main difficulties that we need to overcome.
	
\section{The Converse to $L^2$ Division}\label{sec:inverse-division}

This section is devoted to proving Theorem \ref{thm:main-inverse}.

\subsection{The second fundamental form}
To prove Theorem
\ref{thm:main-inverse}, we examine the second fundamental form in a specific
case. Consider a bounded domain $D \subset \mathbb{C}^n$ and the trivial bundle
$\underline{\mathbb{C}}^{r + 1}$ over $D$, equipped with the standard Euclidean metric,
where $1 \leqslant r \leqslant n$ is an integer. Let $g = (g_1, \ldots, g_{r +
1}) \in (\mathcal{O} (D))^{\oplus (r + 1)}$ whose components have no common zeros, and consider the following exact
sequence
\begin{equation}
  0 \rightarrow S \rightarrow \underline{\mathbb{C}}^{r + 1} \xrightarrow{g}
  \underline{\mathbb{C}} \rightarrow 0, \label{equ:sec4-exact-sequ}
\end{equation}
where $\underline{\mathbb{C}}$ represents the trivial line bundle over $D$, $g$
sends $(h_1, \ldots, h_{r + 1})$ to $g \cdot h = g_1 h_1 + \cdots + g_{r + 1}
h_{r + 1}$ and $S=\operatorname{Ker}g$. The adjoint of $g$ is
 \[g^{\ast} = \left(\frac{\bar{g}_1 }{ | g |^2}, \ldots, \frac{\bar{g}_{r+1}}{ | g |^2}\right).\]
 By \eqref{equ:chern-adjoint}, one has ${D'_S}^{\ast}=\partial^{\ast}$, where $\partial$ is the usual $(1,0)$ part of the exterior derivative $d$ on $\underline{\mathbb{C}}^{r+1}$. It is easy to check that $\partial^{\ast}=-*\bar{\partial}*$ is $\mathcal{O}(D)$-linear, where $*$ is the star operator.

 Hereafter, we assume that the origin $0 = (0, 0, \ldots, 0)
  \in D$. Choose
\[ g_1 = z_1, \ \ldots,\  g_r = z_r,\  g_{r + 1} = z_n + M \]
where $M$ is a large constant. We next compute the second fundamental form in this particular
scenario.

\begin{lemma}
  \label{lem:2nd-form-compute}Under the above setting, one has

    \[ \left.\beta^{\ast}  \right|_0 = \left( \frac{1}{M^2}    d \bar{z}_1,
       \ldots, \frac{1}{M^2}    d \bar{z}_r, 0 \right) \]
    and
    \[ \left.| \partial^{\ast} (\beta^{\ast}\wedge d z) |^2 \right|_0 = \frac{2 r}{M^6}. \]

\end{lemma}

\begin{proof}
  Since $g^{\ast} = (\bar{g}_1 / | g |^2, \ldots, \bar{g}_{r+1} / | g
  |^2)$,   one has
  \begin{equation}
    \bar{\partial} \left( \frac{\bar{g}_j}{| g |^2} \right) = \frac{d
    \bar{z}_j}{| g |^2} - \frac{\bar{z}_j \left( \sum_{l = 1}^r z_l d
    \bar{z}_l + (z_n + M) d \bar{z}_n \right)}{| g |^4}, \label{equ:2nd-form1}
  \end{equation}
  for $1 \leqslant j
  \leqslant r$ and
  \begin{equation}
    \bar{\partial} \left( \frac{\bar{g}_{r + 1}}{| g |^2} \right) = \frac{d
    \bar{z}_n}{| g |^2} - \frac{(\bar{z}_n + M) \left( \sum_{l = 1}^r z_l d
    \bar{z}_l + (z_n + M) d \bar{z}_n \right)}{| g |^4} .
    \label{equ:2nd-form2}
  \end{equation}
  It follows from Proposition \ref{prop:gc-connection} that
  \[ \beta^{\ast}  |_0 = \left( \frac{1}{M^2}  d \bar{z}_1, \ldots,
     \frac{1}{M^2}  d \bar{z}_r, 0 \right) . \]
  For the second equality, we compute $\left.| \partial^{\ast}
  (\beta^{\ast}\wedge dz ) |^2 \right|_0$ separately in the cases $1 \leqslant r \leqslant n - 1$ and
  $r = n$.

  If $1 \leqslant r \leqslant n - 1$ and $1 \leqslant j \leqslant r$, one has
  \begin{eqnarray*}
    &  & \left| \partial^{\ast} \left( \bar{\partial} \left(
    \frac{\bar{g}_j}{| g |^2} \right) \wedge d z \right) \right|^2\\
    & = & \left| \partial^{\ast} \left( \frac{d \bar{z}_j \wedge d z}{| g
    |^2} - \frac{\bar{z}_j \left( \sum_{l = 1}^r z_l d \bar{z}_l + (z_n + M) d
    \bar{z}_n \right) \wedge d z}{| g |^4} \right) \right|^2\\
    & = & \left| \partial^{\ast} \left( \frac{| g |^2 - | z_j |^2}{| g |^4} d
    \bar{z}_j \wedge d z - \frac{\sum_{l \neq j, n} \bar{z}_j z_l}{| g |^4} d
    \bar{z}_l \wedge d z - \frac{\bar{z}_j (z_n + M)}{| g |^4} d \bar{z}_n
    \wedge d z \right) \right|^2\\
    & = & \left\{ \sum_{s = 1}^n \left[ \left| \frac{\partial}{\partial
    \bar{z}_s} \left( \frac{| g |^2 - | z_j |^2}{| g |^4} \right) \right|^2 +
    \right. \right.\\
    &  & \left. \left. + \sum_{l \neq j, n} \left| \frac{\partial}{\partial
    \bar{z}_s} \left( \frac{\bar{z}_j z_l}{| g |^4} \right) \right|^2 + \left|
    \frac{\partial}{\partial \bar{z}_s} \left( \frac{\bar{z}_j (z_n + M)}{| g
    |^4} \right) \right|^2 \right] \right\}.
  \end{eqnarray*}
  It then follows that
  \begin{eqnarray*}
    &  & \frac{\partial}{\partial \bar{z}_s} \left( \frac{| g |^2 - | z_j
    |^2}{| g |^4} \right)\\
    & = & \frac{\partial}{\partial \bar{z}_s} (| g |^2 - | z_j |^2)
    \frac{1}{| g |^4} + (| g |^2 - | z_j |^2) \frac{\partial}{\partial
    \bar{z}_s} \left( \frac{1}{| g |^4} \right)\\
    & = & \left\{\begin{array}{ll}
      \left( z_s - z_j \frac{\partial \bar{z}_j}{\partial \bar{z}_s} \right)
      \frac{1}{| g |^4} + (| g |^2 - | z_j |^2) \cdot \left( \frac{- 2 z_s}{|
      g |^6} \right), & (s \neq n) ;\\
      (z_n + M) \frac{1}{| g |^4} - (| g |^2 - | z_j |^2) \frac{2 (z_n + M)}{|
      g |^6}, & (s = n) .
    \end{array}\right.
  \end{eqnarray*}
  One also has
  \begin{eqnarray*}
    &  & \frac{\partial}{\partial \bar{z}_s} \left( \frac{\bar{z}_j z_l}{| g
    |^4} \right)\\
    & = & \frac{z_l}{| g |^4} \frac{\partial \bar{z}_j}{\partial \bar{z}_s} +
    \bar{z}_j z_l \frac{\partial}{\partial \bar{z}_s} \left( \frac{1}{| g |^4}
    \right)\\
    & = & \left\{\begin{array}{lll}
      \frac{z_l}{| g |^4} \frac{\partial \bar{z}_j}{\partial \bar{z}_s} -
      \bar{z}_j z_l \cdot \frac{2 z_s}{| g |^6}, & (s \neq n) ; & \\
      - \bar{z}_j z_l \frac{2 (z_n + M)}{| g |^6}, & (s = n) , &
    \end{array}\right.
  \end{eqnarray*}
  and
  \begin{eqnarray*}
    &  & \frac{\partial}{\partial \bar{z}_s} \left( \frac{\bar{z}_j (z_n +
    M)}{| g |^4} \right)\\
    & = & \frac{\partial}{\partial \bar{z}_s} (\bar{z}_j (z_n + M))
    \frac{1}{| g |^4} + \bar{z}_j (z_n + M) \frac{\partial}{\partial
    \bar{z}_s} \left( \frac{1}{| g |^4} \right)\\
    & = & \left\{\begin{array}{ll}
      \frac{(z_n + M)}{| g |^4} \frac{\partial \bar{z}_j}{\partial \bar{z}_s}
      - \bar{z}_j (z_n + M) \frac{2 z_s}{| g |^6}, & (s \neq n) ;\\
      \frac{\bar{z}_j}{| g |^4} - \bar{z}_j (z_n + M) \frac{2 (z_n + M)}{| g
      |^6}, & (s = n) .
    \end{array}\right.
  \end{eqnarray*}
  For $j = r + 1$, one has
  \begin{eqnarray*}
    &  & \left| \partial^{\ast} \left( \bar{\partial} \left( \frac{\bar{g}_{r
    + 1}}{| g |^2} \right) \wedge d z \right) \right|^2\\
    & = & \left| \partial^{\ast} \left( \frac{d \bar{z}_n \wedge d z}{| g
    |^2} - \frac{(\bar{z}_n + M) \left( \sum_{l = 1}^r z_l d \bar{z}_l + (z_n
    + M) d \bar{z}_n \right) \wedge d z}{| g |^4} \right) \right|^2\\
    & = & \left| \partial^{\ast} \left( \frac{ | g |^2 - | z_n + M |^2}{| g
    |^4} d \bar{z}_n \wedge d z - \frac{(\bar{z}_n + M) \left( \sum_{l = 1}^r
    z_l d \bar{z}_l \right) \wedge d z}{| g |^4} \right) \right|^2\\
    & = & \sum_{s = 1}^n \left[ \left| \frac{\partial}{\partial \bar{z}_s}
    \left( \frac{ | g |^2 - | z_n + M |^2}{| g |^4} \right) \right|^2 +
    \sum_{l = 1}^r \left| \frac{\partial}{\partial \bar{z}_s} \left(
    \frac{(\bar{z}_n + M) z_l}{| g |^4} \right) \right|^2 \right]
    .
  \end{eqnarray*}
  It then follows that
  \begin{eqnarray*}
    &  & \frac{\partial}{\partial \bar{z}_s} \left( \frac{ | g |^2 - | z_n +
    M |^2}{| g |^4} \right)\\
    & = & \frac{\partial}{\partial \bar{z}_s} (| g |^2 - | z_n + M |^2)
    \frac{1}{| g |^4} + (| g |^2 - | z_n + M |^2) \frac{\partial}{\partial
    \bar{z}_s} \left( \frac{1}{| g |^4} \right)\\
    & = & \left\{\begin{array}{ll}
      \frac{z_s}{| g |^4} - (| g |^2 - | z_n + M |^2) \frac{2 z_s}{| g |^6}, &
      (s \neq n) ;\\
      - (| g |^2 - | z_n + M |^2) \frac{2 (z_n + M)}{| g |^6}, & (s = n),
    \end{array}\right.
  \end{eqnarray*}
  and
  \begin{eqnarray*}
    &  & \frac{\partial}{\partial \bar{z}_s} \left( \frac{(\bar{z}_n + M)
    z_l}{| g |^4} \right)\\
    & = & \frac{\partial}{\partial \bar{z}_s} ((\bar{z}_n + M) z_l)
    \frac{1}{| g |^4} + (\bar{z}_n + M) z_l \frac{\partial}{\partial
    \bar{z}_s} \left( \frac{1}{| g |^4} \right)\\
    & = & \left\{\begin{array}{ll}
      - (\bar{z}_n + M) z_l \frac{2 z_s}{| g |^6}, & (s \neq n) ;\\
      \frac{z_l}{| g |^4} - (\bar{z}_n + M) z_l \frac{2 (z_n + M)}{| g |^6}, &
      (s = n) .
    \end{array}\right.
  \end{eqnarray*}

  Therefore, at $0$, one has
  \begin{eqnarray*}
    \left.| \partial^{\ast} (\beta^{\ast} \wedge d z) |^2 \right|_0 & = & \sum_{j = 1}^{r
    + 1} \left.\left| \partial^{\ast} \left( \bar{\partial} \left(
    \frac{\bar{g}_j}{| g |^2} \right) \wedge d z \right) \right|^2 \right|_0\\
    & = & \sum_{j = 1}^r \frac{2}{M^6}\\
    & = & \frac{2 r}{M^6} .
  \end{eqnarray*}
  If $r = n$, for $1 \leqslant j \leqslant n - 1$, one has
  \begin{eqnarray*}
    &  & \left| \partial^{\ast} \left( \bar{\partial} \left(
    \frac{\bar{g}_j}{| g |^2} \right) \wedge d z \right) \right|^2\\
    & = & \left| \partial^{\ast} \left( \frac{d \bar{z}_j \wedge d z}{| g
    |^2} - \frac{\bar{z}_j \left( \sum_{l = 1}^{n - 1} z_l d \bar{z}_l + (2
    z_n + M) d \bar{z}_n \right) \wedge d z}{| g |^4} \right) \right|^2\\
    & = & \left| \partial^{\ast} \left( \frac{| g |^2 - | z_j |^2}{| g |^4} d
    \bar{z}_j \wedge d z - \frac{\sum_{l \neq j, n} \bar{z}_j z_l}{| g |^4} d
    \bar{z}_l \right. \wedge d z \right. \\
    &  & \left. \left. - \frac{\bar{z}_j (2 z_n + M)}{| g |^4} d \bar{z}_n
    \wedge d z \right) \right|^2\\
    & = & \sum_{s = 1}^n \left[ \left| \frac{\partial}{\partial \bar{z}_s}
    \left( \frac{| g |^2 - | z_j |^2}{| g |^4} \right) \right|^2 + \right.\\
    &  & \left. + \sum_{l \neq j, n} \left| \frac{\partial}{\partial
    \bar{z}_s} \left( \frac{\bar{z}_j z_l}{| g |^4} \right) \right|^2 + \left|
    \frac{\partial}{\partial \bar{z}_s} \left( \frac{\bar{z}_j (2 z_n + M)}{|
    g |^4} \right) \right|^2 \right] .
  \end{eqnarray*}
  It then follows that
  \begin{eqnarray*}
    &  & \frac{\partial}{\partial \bar{z}_s} \left( \frac{| g |^2 - | z_j
    |^2}{| g |^4} \right)\\
    & = & \frac{\partial}{\partial \bar{z}_s} (| g |^2 - | z_j |^2)
    \frac{1}{| g |^4} + (| g |^2 - | z_j |^2) \frac{\partial}{\partial
    \bar{z}_s} \left( \frac{1}{| g |^4} \right)\\
    & = & \left\{\begin{array}{ll}
      \left( z_s - z_j \frac{\partial \bar{z}_j}{\partial \bar{z}_s} \right)
      \frac{1}{| g |^4} + (| g |^2 - | z_j |^2) \cdot \left( \frac{- 2 z_s}{|
      g |^6} \right), & (s \neq n) ;\\
      (2 z_n + M) \frac{1}{| g |^4} - (| g |^2 - | z_j |^2) \frac{2 (2 z_n +
      M)}{| g |^6}, & (s = n) .
    \end{array}\right.
  \end{eqnarray*}
  One also has
  \begin{eqnarray*}
    &  & \frac{\partial}{\partial \bar{z}_s} \left( \frac{\bar{z}_j z_l}{| g
    |^4} \right)\\
    & = & \frac{z_l}{| g |^4} \frac{\partial \bar{z}_j}{\partial \bar{z}_s} +
    \bar{z}_j z_l \frac{\partial}{\partial \bar{z}_s} \left( \frac{1}{| g |^4}
    \right)\\
    & = & \left\{\begin{array}{lll}
      \frac{z_l}{| g |^4} \frac{\partial \bar{z}_j}{\partial \bar{z}_s} -
      \bar{z}_j z_l \cdot \frac{2 z_s}{| g |^6}, & (s \neq n) ; & \\
      - \bar{z}_j z_l \frac{2 (2 z_n + M)}{| g |^6}, & (s = n) , &
    \end{array}\right.
  \end{eqnarray*}
  and
  \begin{eqnarray*}
    &  & \frac{\partial}{\partial \bar{z}_s} \left( \frac{\bar{z}_j (2 z_n +
    M)}{| g |^4} \right)\\
    & = & \frac{\partial}{\partial \bar{z}_s} (\bar{z}_j (2 z_n + M))
    \frac{1}{| g |^4} + \bar{z}_j (2 z_n + M) \frac{\partial}{\partial
    \bar{z}_s} \left( \frac{1}{| g |^4} \right)\\
    & = & \left\{\begin{array}{ll}
      \frac{(2 z_n + M)}{| g |^4} \frac{\partial \bar{z}_j}{\partial
      \bar{z}_s} - \bar{z}_j (2 z_n + M) \frac{2 z_s}{| g |^6}, & (s \neq n)
      ;\\
      - \bar{z}_j (2 z_n + M) \frac{2 (2 z_n + M)}{| g |^6}, & (s = n) .
    \end{array}\right.
  \end{eqnarray*}
  For $j = n$, one has
  \begin{eqnarray*}
    &  & \left| \partial^{\ast} \left( \bar{\partial} \left(
    \frac{\bar{g}_n}{| g |^2} \right) \wedge d z \right) \right|^2\\
    & = & \left| \partial^{\ast} \left( \frac{d \bar{z}_n \wedge d z}{| g
    |^2} - \frac{\bar{z}_n \left( \sum_{l = 1}^{n - 1} z_l d \bar{z}_l + (2
    z_n + M) d \bar{z}_n \right) \wedge d z}{| g |^4} \right) \right|^2\\
    & = & \left| \partial^{\ast} \left( \frac{ | g |^2 - \bar{z}_n (2 z_n +
    M)}{| g |^4} d \bar{z}_n \wedge d z - \frac{\bar{z}_n \left( \sum_{l =
    1}^{n - 1} z_l d \bar{z}_l \right) \wedge d z}{| g |^4} \right)
    \right|^2\\
    & = & \sum_{s = 1}^n \left[ \left| \frac{\partial}{\partial \bar{z}_s}
    \left( \frac{ | g |^2 - \bar{z}_n (2 z_n + M)}{| g |^4} \right) \right|^2
    + \sum_{l = 1}^{n - 1} \left| \frac{\partial}{\partial \bar{z}_s} \left(
    \frac{\bar{z}_n z_l}{| g |^4} \right) \right|^2 \right] .
  \end{eqnarray*}
  It then follows that
  \begin{eqnarray*}
    &  & \frac{\partial}{\partial \bar{z}_s} \left( \frac{ | g |^2 -
    \bar{z}_n (2 z_n + M)}{| g |^4} \right)\\
    & = & \frac{\partial}{\partial \bar{z}_s} (| g |^2 - \bar{z}_n (2 z_n +
    M)) \frac{1}{| g |^4}\\
    &  & + (| g |^2 - \bar{z}_n (2 z_n + M)) \frac{\partial}{\partial
    \bar{z}_s} \left( \frac{1}{| g |^4} \right)\\
    & = & \left\{\begin{array}{ll}
      \frac{z_s}{| g |^4} - (| g |^2 - \bar{z}_n (2 z_n + M)) \frac{2 z_s}{| g
      |^6}, & (s \neq n) ;\\
      - (| g |^2 - \bar{z}_n (2 z_n + M)) \frac{2 (2 z_n + M)}{| g |^6}, & (s =
      n) ,
    \end{array}\right.
  \end{eqnarray*}
  and
  \begin{eqnarray*}
    &  & \frac{\partial}{\partial \bar{z}_s} \left( \frac{\bar{z}_n z_l}{| g
    |^4} \right)\\
    & = & \frac{\partial}{\partial \bar{z}_s} (\bar{z}_n z_l) \frac{1}{| g
    |^4} + \bar{z}_n z_l \frac{\partial}{\partial \bar{z}_s} \left( \frac{1}{|
    g |^4} \right)\\
    & = & \left\{\begin{array}{ll}
      - \bar{z}_n z_l \frac{2 z_s}{| g |^6}, & (s \neq n) ;\\
      \frac{z_l}{| g |^4} - \bar{z}_n z_l \frac{2 (2 z_n + M)}{| g |^6}, & (s
      = n) .
    \end{array}\right.
  \end{eqnarray*}
  For $j = n + 1$, one has
  \begin{eqnarray*}
    &  & \left| \partial^{\ast} \left( \bar{\partial} \left( \frac{\bar{g}_{n
    + 1}}{| g |^2} \right) \wedge d z \right) \right|^2\\
    & = & \left| \partial^{\ast} \left( \frac{d \bar{z}_n \wedge d z}{| g
    |^2} - \frac{(\bar{z}_n + M) \left( \sum_{l = 1}^{n - 1} z_l d \bar{z}_l +
    (2 z_n + M) d \bar{z}_n \right) \wedge d z}{| g |^4} \right) \right|^2\\
    & = & \left| \partial^{\ast} \left( \frac{ | g |^2 - (\bar{z}_n + M) (2
    z_n + M)}{| g |^4} d \bar{z}_n \wedge d z - \frac{(\bar{z}_n + M) \left(
    \sum_{l = 1}^{n - 1} z_l d \bar{z}_l \right) \wedge d z}{| g |^4} \right)
    \right|^2\\
    & = & \sum_{s = 1}^n \left[ \left| \frac{\partial}{\partial \bar{z}_s}
    \left( \frac{ | g |^2 - (\bar{z}_n + M) (2 z_n + M)}{| g |^4} \right)
    \right|^2 \right.\\
    &  & \left. + \sum_{l = 1}^{n - 1} \left| \frac{\partial}{\partial
    \bar{z}_s} \left( \frac{(\bar{z}_n + M) z_l}{| g |^4} \right) \right|^2
    \right] .
  \end{eqnarray*}
  It then follows that
  \begin{eqnarray*}
    &  & \frac{\partial}{\partial \bar{z}_s} \left( \frac{ | g |^2 -
    (\bar{z}_n + M) (2 z_n + M)}{| g |^4} \right)\\
    & = & \frac{\partial}{\partial \bar{z}_s} (| g |^2 - (\bar{z}_n + M) (2
    z_n + M)) \frac{1}{| g |^4}\\
    &  & + (| g |^2 - (\bar{z}_n + M) (2 z_n + M)) \frac{\partial}{\partial
    \bar{z}_s} \left( \frac{1}{| g |^4} \right)\\
    & = & \left\{\begin{array}{ll}
      \frac{z_s}{| g |^4} - (| g |^2 - (\bar{z}_n + M) (2 z_n + M)) \frac{2
      z_s}{| g |^6}, & (s \neq n) ;\\
      - (| g |^2 - (\bar{z}_n + M) (2 z_n + M)) \frac{2 (2 z_n + M)}{| g |^6},
      & (s = n) ,
    \end{array}\right.
  \end{eqnarray*}
  and
  \begin{eqnarray*}
    &  & \frac{\partial}{\partial \bar{z}_s} \left( \frac{(\bar{z}_n + M)
    z_l}{| g |^4} \right)\\
    & = & \frac{\partial}{\partial \bar{z}_s} ((\bar{z}_n + M) z_l)
    \frac{1}{| g |^4} + (\bar{z}_n + M) z_l \frac{\partial}{\partial
    \bar{z}_s} \left( \frac{1}{| g |^4} \right)\\
    & = & \left\{\begin{array}{ll}
      - (\bar{z}_n + M) z_l \frac{2 z_s}{| g |^6}, & (s \neq n) ;\\
      \frac{z_l}{| g |^4} - (\bar{z}_n + M) z_l \frac{2 (2 z_n + M)}{| g |^6},
      & (s = n) .
    \end{array}\right.
  \end{eqnarray*}
  Therefore, at $0$, one also has
  \begin{eqnarray*}
    \left.| \partial^{\ast} (\beta^{\ast} \wedge d z) |^2 \right|_0 & = & \sum_{j =
    1}^{n + 1} \left.\left| \partial^{\ast} \left( \bar{\partial} \left(
    \frac{\bar{g}_j}{| g |^2} \right) \wedge d z \right) \right|^2 \right|_0\\
    & = & \sum_{j = 1}^n \frac{2}{M^6}
    = \frac{2 n}{M^6} \\ & = & \frac{2 r}{M^6}  .
  \end{eqnarray*}

\end{proof}

\subsection{The converse to $L^2$ division}

We give the proof of Theorem \ref{thm:main-inverse} in this subsection. In
the light of Proposition \ref{prop:dbar-division-equivalent}, it suffices to
prove the following proposition.

\begin{proposition}
  \label{prop:main}Let $D \subset \mathbb{C}^n$ be a bounded domain,
  $\varphi \in C^2 (D)$ and $r$ an integer between $1$ and $n$.
  If there exists a non-zero holomorphic function $f\in A^2(D,\varphi)$ such that for any
   $g = (g_1, \ldots, g_{r + 1}) \in (\mathcal{O} (D))^{\oplus (r +
  1)}$ whose components have no common zeros and
  $d g_1\wedge\cdots\wedge d g_r \neq 0$, and any smooth $\psi\in PSH(D)$ satisfying
  \[ i \partial \bar{\partial} \psi \otimes
     \tmop{Id}_{\underline{\mathbb{C}}^{r + 1}} + i \beta^{\ast} \wedge \beta
     > 0, \]
   with
  \[ \int_D \left( \left[ i \partial \bar{\partial} \psi \otimes
     \tmop{Id}_{\underline{\mathbb{C}}^{r + 1}} + i \beta^{\ast} \wedge \beta,
     \Lambda \right]^{- 1} \beta^{\ast} \wedge d z, \beta^{\ast}\wedge d z  \right)|f|^2 e^{- \varphi
     - \psi} d V < \infty, \]
  there exists an $S$-valued section $u$ $S$ such that
  \[ \bar{\partial} u = \beta^{\ast} f \]
  and
  \begin{equation}
    \int_D | u |^2 e^{- \varphi - \psi} d V \leqslant \int_D \left( \left[ i
    \partial \bar{\partial} \psi \otimes \tmop{Id}_{\underline{\mathbb{C}}^{r
    + 1}} + i \beta^{\ast} \wedge \beta, \Lambda \right]^{- 1} \beta^{\ast} \wedge d z,
    \beta^{\ast} \wedge d z \right) |f|^2e^{- \varphi - \psi} d V,
    \label{equ:skoda-estimate}
  \end{equation}
then the sum of any $r$ eigenvalues of $i \partial
  \bar{\partial} \varphi$ is non-negative.
\end{proposition}

\begin{proof}
  We argue by contradiction.
  Suppose that at some point $z_0\in D$, there exist $r$ eigenvalues $\lambda_1, \lambda_2,
  \ldots, \lambda_r$ of $i \partial \bar{\partial} \varphi$ such that
  \[ \lambda_1 + \lambda_2 + \cdots + \lambda_r < 0. \]
  Without loss of generality, we may assume that $z_0$ is the origin.
  Denote by $\xi_1, \xi_2, \ldots, \xi_r$ the unit eigenvectors corresponding
  to $\lambda_1, \lambda_2, \ldots, \lambda_r$ respectively. After a linear transformation, we may assume that
  \[ \xi_j = (0, \ldots, 1, \ldots, 0), \]
  i.e., the $j$-th component of $\xi_j$ is $1$ and the other components of $\xi_j$
  are zero. Regarding $\xi_k$
  $(1 \leqslant k \leqslant r)$ as $(n, 1)$ forms, it follows from \eqref{equ:curvature-operator} that, when acting on $(n, 1)$ forms, the
  eigenvalues of $[i \partial \bar{\partial} \varphi, \Lambda]$ coincide with
  those of $i \partial \bar{\partial} \varphi$. Let $e_1,
  \ldots, e_{r+1}$ be an orthonormal basis of $\underline{\mathbb{C}}^{r+1}$. Then
  \begin{eqnarray*}
    &  & \left\langle \left[ i \partial \bar{\partial} \varphi \otimes
    \tmop{Id}_{\underline{\mathbb{C}}^{r + 1}}, \Lambda \right] \sum_{j = 1}^r
    \xi_j \otimes e_j, \sum_{j = 1}^r \xi_j \otimes e_j \right\rangle\\
    & = & \sum_{j = 1}^r \langle [i \partial \bar{\partial} \varphi, \Lambda]
    \xi_j, \xi_j \rangle\\
    & = & \sum^r_{j = 1} \lambda_j\\
    & < & 0.
  \end{eqnarray*}
  Take
  \[ g_1 = z_1,\  g_2 = z_2, \ldots,\  g_r = z_r,\  g_{r + 1} = z_n + M, \]
  where $M$ is a real constant such that $M > \sup_{z \in D} | z |$.
   Since $| g |^2$ has
  no zeros in $\overline{D}$ and
  \[ \beta^{\ast} = \left( \bar{\partial} \left( \frac{\bar{g}_1}{| g |^2}
     \right), \ldots, \bar{\partial} \left( \frac{\bar{g}_{r + 1}}{| g |^2}
     \right) \right), \]
      it
  follows that every component of $\beta^{\ast}$ is smooth on $\overline{D}$.
   Therefore, we may define the nonnegative number
  \begin{equation}
    k_0 = \inf \left\{ c > 0;\quad  c\cdot i \partial \bar{\partial}   | z |^2 \otimes
    \tmop{Id}_{\underline{\mathbb{C}}^{r + 1}} + i \beta^{\ast} \wedge \beta >
    0 \ \text{on}\  D\right\} \label{equ:def-k0}
  \end{equation}
   and set $k_{\varepsilon} = k_0 + \varepsilon$ for an arbitrary
  constant $\varepsilon > 0$.

  We choose $\psi = (k + k_{\varepsilon}) (| z |^2 - {\delta}^2 / 4)$, where $k$ and
  ${\delta}$ are positive constants. Without loss of generality, we may assume that $f(0)\neq 0$. By our choice of $g$ and $\psi$,
  \[
     \int_D \left( \left[ i \partial \bar{\partial} \psi \otimes
     \tmop{Id}_{\underline{\mathbb{C}}^{r + 1}} + i \beta^{\ast} \wedge \beta,
     \Lambda \right]^{- 1} \beta^{\ast} \wedge d z, \beta^{\ast}\wedge d z  \right)|f|^2 e^{- \varphi
     - \psi} d V < \infty.
  \]
    By hypothesis, one can find an $S$-valued section $u$ such that
  \[ \bar{\partial} u = \beta^{\ast} f \]
  and
  \[ \int_D | u |^2 e^{- \varphi - \psi} d V \leqslant      \int_D \left( \left[ i \partial \bar{\partial} \psi \otimes
     \tmop{Id}_{\underline{\mathbb{C}}^{r + 1}} + i \beta^{\ast} \wedge \beta,
     \Lambda \right]^{- 1} \beta^{\ast} \wedge d z, \beta^{\ast}\wedge d z  \right)|f|^2 e^{- \varphi
     - \psi} d V. \]
  It follows that for any $S$ valued $(n,1)$ form $\alpha$ with compact support in $D$,
  \begin{eqnarray*}
    &  & \left|\llangle \beta^{\ast} f \wedge d z, \alpha \rrangle_{\varphi + \psi}\right|^2\\
    & = & \left|\llangle \bar{\partial} u \wedge d z, \alpha \rrangle_{\varphi + \psi}\right|^2\\
    & = & \left|\llangle u\wedge d z, \bar{\partial}^{\ast}_S \alpha \rrangle_{\varphi +
    \psi}\right|^2\\
    & \leqslant &  \| \bar{\partial}_S^{\ast}\alpha \|^2_{\varphi + \psi}
                \| u \|^2_{\varphi + \psi}\\
    & \leqslant &\| \bar{\partial}_S^{\ast} \alpha \|^2_{\varphi + \psi} \cdot \int_D \left\langle \left[ ((k+ k_{\varepsilon})i  \partial \bar{\partial}   | z |^2) \otimes
    \tmop{Id}_{\underline{\mathbb{C}}^{r + 1}} + i \beta^{\ast} \wedge \beta,
    \Lambda \right]^{- 1} \beta^{\ast} \wedge d z, \beta^{\ast} \wedge d z \right\rangle |f|^2 e^{-
    \varphi - \psi} d V \\
    & < & \| \bar{\partial}_S^{\ast} \alpha \|^2_{\varphi + \psi} \cdot \int_D \left\langle \left[ i k \partial \bar{\partial} | z |^2
    \otimes \tmop{Id}_{\underline{\mathbb{C}}^{r + 1}}, \Lambda \right]^{- 1}
    \beta^{\ast} \wedge d z, \beta^{\ast} \wedge d z \right\rangle |f|^2 e^{- \varphi - \psi} d V\\
    & = &\| \bar{\partial}_S^{\ast} \alpha \|^2_{\varphi + \psi}\cdot \frac{1}{k} \int_D | \beta^{\ast}  |^2 |f|^2 e^{- \varphi - \psi} d V,
  \end{eqnarray*}
  that is,
  \begin{equation}
    \left| \int_D \langle \beta^{\ast} f \wedge d z, \alpha \rangle e^{- \varphi - \psi}
    d V \right|^2 < \frac{1}{k} \int_D | \beta^{\ast}  |^2 |f|^2 e^{- \varphi -
    \psi} d V \cdot \int_D | \bar{\partial}^{\ast}_S \alpha |^2 e^{- \varphi -
    \psi} d V. \label{equ:key1}
  \end{equation}

  Let $K$ and $K'$ be two bounded relatively compact open subsets of $D$ such that $K \Subset K'$ and $0\in K$.
  Choose a  cut-off function $\chi$ such that
  \[
  \begin{cases}
    0 \leqslant \chi \leqslant 1, \\
    \chi = 1 \text{ on } K, \\
    \chi = 0 \text{ on } D \setminus K'.
    \end{cases}
    \]
  We then take
  \[ \alpha = \chi \left[ i k \partial \bar{\partial} | z |^2 \otimes
     \tmop{Id}_{\underline{\mathbb{C}}^{r + 1}}, \Lambda \right]^{- 1}
     \beta^{\ast} f\wedge d z = \chi \frac{1}{k} \beta^{\ast} f\wedge d z. \]
  It follows from \eqref{equ:key1} that
  \begin{eqnarray*}
     \frac{1}{k^2} \left( \int_D \chi | \beta^{\ast}|^2| f |^2 e^{- \varphi
    - \psi} d V \right)^2
     < \frac{1}{k} \int_D | \beta^{\ast}|^2| f |^2 e^{- \varphi - \psi} d V
    \cdot \int_D | \bar{\partial}^{\ast}_S \alpha |^2 e^{- \varphi - \psi} d
    V,
  \end{eqnarray*}
  i.e.
  \begin{equation}
    \frac{1}{k} \frac{\left( \int_D \chi | \beta^{\ast}|^2| f |^2 e^{- \varphi -
    \psi} d V \right)^2}{\int_D | \beta^{\ast}|^2| f |^2 e^{- \varphi - \psi} d V}
    < \int_D | \bar{\partial}^{\ast}_S \alpha |^2 e^{- \varphi - \psi} d V.
    \label{equ:key2}
  \end{equation}
  By the Bochner-Kodaira-Nakano identity,
  \begin{eqnarray*}
    &  & \| \bar{\partial}^{\ast}_S \alpha \|^2_{\varphi + \psi} + \|
    \bar{\partial} \alpha \|^2_{\varphi + \psi}\\
    & = & \left\langle\!\left\langle [i \Theta (S), \Lambda] \alpha, \alpha \right\rangle\!\right\rangle_{\varphi +
    \psi} + \| D'_S \alpha \|^2_{\varphi + \psi} + \| \partial^{\ast}
    \alpha \|^2_{\varphi + \psi}\\
    & = & \left\langle\!\!\left\langle \left[ i (k + k_{\varepsilon}) \partial
    \bar{\partial} | z |^2 \otimes \tmop{Id}_{\underline{\mathbb{C}}^{r + 1}},
    \Lambda \right] \alpha, \alpha \right\rangle\!\!\right\rangle_{\varphi + \psi} +
    \left\langle\!\!\left\langle \left[ i \partial \bar{\partial} \varphi \otimes
    \tmop{Id}_{\underline{\mathbb{C}}^{r + 1}}, \Lambda \right] \alpha, \alpha
    \right\rangle\!\!\right\rangle_{\varphi + \psi}\\
    &  & + \left\langle\!\left\langle [i \beta^{\ast} \wedge \beta, \Lambda] \alpha, \alpha
    \right\rangle\!\right\rangle_{\varphi + \psi} + \| \partial^{\ast} \alpha \|^2_{\varphi +
    \psi} .
  \end{eqnarray*}
  Since $[i \beta^{\ast} \wedge \beta, \Lambda]$ is negative definite on $(n,
  1)$ forms, one has
  \begin{eqnarray*}
    \| \bar{\partial}_S^{\ast} \alpha \|^2_{\varphi + \psi}& \leqslant & \left\langle\!\!\left\langle \left[ i (k + k_{\varepsilon}) \partial
    \bar{\partial} | z |^2 \otimes \tmop{Id}_{\underline{\mathbb{C}}^{r + 1}},
    \Lambda \right] \alpha, \alpha \right\rangle\!\!\right\rangle_{\varphi + \psi}\\
    &  & + \left\langle\!\!\left\langle \left[ i \partial \bar{\partial} \varphi \otimes
    \tmop{Id}_{\underline{\mathbb{C}}^{r + 1}}, \Lambda \right] \alpha, \alpha
    \right\rangle\!\!\right\rangle_{\varphi + \psi}\\
    & & + \| \partial^{\ast} \alpha
    \|^2_{\varphi + \psi},
  \end{eqnarray*}
  that is
  \begin{equation}\label{equ:key4s}
    \begin{split}
       \| \bar{\partial}_S^{\ast} \alpha \|^2_{\varphi + \psi}  \leqslant&  \frac{k + k_{\varepsilon}}{k^2} \int_D | \chi
    \beta^{\ast}|^2| f |^2 e^{- \varphi - \psi} d V \\
         & + \frac{1}{k^2} \int_D \left\langle \left[ i \partial \bar{\partial}
    \varphi \otimes \tmop{Id}_{\underline{\mathbb{C}}^{r + 1}}, \Lambda
    \right] \chi \beta^{\ast} \wedge d z, \chi \beta^{\ast} \wedge d z \right\rangle |f|^2 e^{-
    \varphi - \psi} d V \\
         & + \frac{1}{k^2} \int_D | \partial^{\ast} (\chi \beta^{\ast}\wedge d z )f
    |^2 e^{- \varphi - \psi} d V.
    \end{split}
  \end{equation}
  Combining \eqref{equ:key4s} with \eqref{equ:key2}, one has
  \begin{eqnarray*}
    \frac{1}{k} \frac{\left( \int_D \chi | \beta^{\ast} |^2|f |^2 e^{-
    \varphi - \psi} d V \right)^2}{\int_D | \beta^{\ast}|^2| f |^2 e^{- \varphi -
    \psi} d V}& < & \frac{k + k_{\varepsilon}}{k^2} \int_D | \chi \beta^{\ast}|^2| f |^2
    e^{- \varphi - \psi} d V\\
    &  & + \frac{1}{k^2} \int_D \left\langle \left[ i \partial \bar{\partial}
    \varphi \otimes \tmop{Id}_{\underline{\mathbb{C}}^{r + 1}}, \Lambda
    \right] \chi \beta^{\ast} \wedge d z, \chi \beta^{\ast} \wedge d z \right\rangle|f|^2 e^{-
    \varphi - \psi} d V\\
    &  & + \frac{1}{k^2} \int_D | \partial^{\ast} (\chi \beta^{\ast} \wedge d z)f
    |^2 e^{- \varphi - \psi} d V.
  \end{eqnarray*}
  It follows that
  \begin{eqnarray}
    &  & \int_D | \partial^{\ast} (\chi \beta^{\ast} \wedge d z)f |^2 e^{- \varphi
    - \psi} d V + \int_D \left\langle \left[ i \partial \bar{\partial} \varphi
    \otimes \tmop{Id}_{\underline{\mathbb{C}}^{r + 1}}, \Lambda \right] \chi
    \beta^{\ast} \wedge d z, \chi \beta^{\ast} \wedge d z \right\rangle |f|^2 e^{- \varphi - \psi} d
    V \nonumber\\
    & > & \frac{k \left( \int_D \chi | \beta^{\ast}|^2| f |^2 e^{- \varphi -
    \psi} d V \right)^2}{\int_D | \beta^{\ast}|^2| f |^2 e^{- \varphi - \psi} d V}
    - (k + k_{\varepsilon}) \int_D  | \chi \beta^{\ast} |^2|f |^2 e^{- \varphi -
    \psi} d V  \label{key-estimate}\\
    & \geqslant & \frac{k \int_D \chi | \beta^{\ast} |^2|f |^2 e^{- \varphi -
    \psi} d V \left( \int_D \chi | \beta^{\ast}|^2| f |^2 e^{- \varphi - \psi} d
    V - \int_D | \beta^{\ast} |^2|f |^2 e^{- \varphi - \psi} d V \right)}{\int_D |
    \beta^{\ast}|^2| f |^2 e^{- \varphi - \psi} d V} \nonumber\\
    &  & - k_{\varepsilon} \int_D | \chi \beta^{\ast}|^2| f |^2 e^{- \varphi -
    \psi} d V. \nonumber
  \end{eqnarray}
  Take $\delta$ small enough such that $B(0,{\delta}/{2})$ is contained in $K$.
  Since
  \[ \lim_{k \rightarrow \infty} e^{- \psi + \log k} = \lim_{k \rightarrow \infty} e^{
      - (k + k_{\varepsilon}) (| z |^2 - {\delta}^2 / 4) + \log k} = 0 \]
  for $z \in D \setminus B (0, {\delta} / 2)$, we can choose $k$ large enough such that
  \begin{eqnarray*}
    &  & k \left( \int_D \chi | \beta^{\ast}|^2| f |^2 e^{- \varphi - \psi} d V
    - \int_D | \beta^{\ast}|^2| f |^2 e^{- \varphi - \psi} d V \right)\\
    & = & \int_D (\chi - 1) | \beta^{\ast}|^2| f |^2 e^{- \varphi - \psi + \log
    k} d V\\
    & = & \int_{D\setminus B (0, \delta / 2)} (\chi - 1) | \beta^{\ast}|^2| f |^2
    e^{- \varphi - \psi + \log k} d V\\
    & > & - \varepsilon.
  \end{eqnarray*}
  Moreover, since $e^{\psi}\to 0$ as $k\to\infty$ on $D\setminus B(0,{\delta}/{2})$, we can 
  also assume that $k$ is large enough such that
  \begin{eqnarray*}
    &  & \int_{D \setminus B (0, \delta / 2)} | \partial^{\ast} (\chi \beta^{\ast}
    \wedge d z)f |^2 e^{- \varphi - \psi} d V < \frac{\varepsilon}{3},\\
    &  & \left|\int_{D \setminus B (0, \delta / 2)} \left\langle \left[ i \partial
    \bar{\partial} \varphi \otimes \tmop{Id}_{\underline{\mathbb{C}}^{r + 1}},
    \Lambda \right] \chi \beta^{\ast} \wedge d z, \chi \beta^{\ast} \wedge d z \right\rangle |f|^2
    e^{- \varphi - \psi} d V\right| < \frac{\varepsilon}{3},\\
    &  & k_{\varepsilon} \int_{D \setminus B (0, \delta / 2)} | \chi \beta^{\ast}|^2 |f|^2
    e^{- \varphi - \psi} d V < \frac{\varepsilon}{3}.
  \end{eqnarray*}
  Decomposing the integrals in \eqref{key-estimate} into
  $D \setminus B(0,\delta/2)$ and $B(0,\delta/2)$, and using that $\chi=1$
  on $B(0,\delta/2)$, we therefore obtain
  \begin{eqnarray}
    &  & \int_{B (0, \delta / 2)} \left\langle \left[ i \partial \bar{\partial}
    \varphi \otimes \tmop{Id}_{\underline{\mathbb{C}}^{r + 1}}, \Lambda
    \right] \beta^{\ast} \wedge d z, \beta^{\ast} \wedge d z \right\rangle |f|^2 e^{- \varphi - \psi}
    d V \nonumber\\
    &  & + \int_{B (0, \delta / 2)} | \partial^{\ast} (\beta^{\ast}\wedge d z )f |^2 e^{- 
    \varphi - \psi} d V \nonumber\\
    & > & - 2\varepsilon - k_{\varepsilon} \int_{B (0, \delta / 2)} | \beta^{\ast}|^2|
    f |^2 e^{- \varphi - \psi} d V.  \label{equ:key-estimate1}
  \end{eqnarray}
  It follows from Lemma \ref{lem:2nd-form-compute} that
  \[ \beta^{\ast}  |_0 = \left( \frac{1}{M^2}   d \bar{z}_1,
     \ldots, \frac{1}{M^2}   d \bar{z}_r, 0 \right), \]
  thus we can choose $\delta > 0$ small enough such that on $B (0, \delta)$,
  \begin{equation}
    | \beta^{\ast}|^2| f |^2 < \frac{2 r}{M^4} | f (0) |^2 \label{equ:estimate1}.
  \end{equation}
  Since at the origin $0$,
  \begin{equation*}
    \left\langle \left[ i \partial \bar{\partial} \varphi \otimes
    \tmop{Id}_{\underline{\mathbb{C}}^{r + 1}}, \Lambda \right] \chi
    \beta^{\ast} \wedge d z, \chi \beta^{\ast} \wedge d z \right\rangle |f|^2
    =  \left(\sum_{j=1}^{r}\lambda_j\right)\frac{1}{M^4} | f (0) |^2,
  \end{equation*}
  we may also assume that on $B (0, \delta)$,
  \begin{equation}
    \left\langle \left[ i \partial \bar{\partial} \varphi \otimes
    \tmop{Id}_{\underline{\mathbb{C}}^{r + 1}}, \Lambda \right] \chi
    \beta^{\ast} \wedge d z, \chi \beta^{\ast} \wedge d z \right\rangle |f|^2
    <  \left(\sum_{j=1}^{r}\lambda_j\right)\frac{2}{M^4} | f (0) |^2.  \label{equ:neg-estimate}
  \end{equation}

  On $B (0, \delta)$, we have $\partial^{\ast} (\chi \beta^{\ast} f\wedge d z)
  = \partial^{\ast} (\beta^{\ast} f\wedge d z)$.
  It follows from Lemma \ref{lem:2nd-form-compute} that for all $1 \leqslant r
  \leqslant n$,
  \begin{eqnarray*}
    \left.| \partial^{\ast} (\beta^{\ast} f \wedge d  z) |^2 \right|_0 & = & \sum_{j = 1}^{r + 1}
    \left.\left| \partial^{\ast} \left( \bar{\partial} \left( \frac{\bar{g}_j}{| g
    |^2} \right) \wedge d z \right) f\right|^2 \right|_0\\
    & = & \sum_{j = 1}^r \frac{2}{M^6} | f (0) |^2\\
    & = &   \frac{2 r}{M^6} | f (0) |^2.
  \end{eqnarray*}
  Therefore, by shrinking $B (0, \delta)$ if necessary, we may always assume that
  on $B (0, \delta)$,
  \[ | \partial^{\ast} (\chi \beta^{\ast}\wedge d z)f |^2 \leqslant \frac{3 r}{M^6}
     | f (0) |^2 . \]
  Combining the computations above, \eqref{equ:neg-estimate} and
  \eqref{equ:estimate1}, it follows that
  \begin{eqnarray*}
    &  & \int_{B (0, \delta / 2)} \left\langle \left[ i \partial \bar{\partial}
    \varphi \otimes \tmop{Id}_{\underline{\mathbb{C}}^{r + 1}}, \Lambda
    \right] \beta^{\ast} \wedge dz, \beta^{\ast} \wedge dz \right\rangle |f|^2 e^{- \varphi - \psi}
    d V\\
    & < & \left(\sum_{j=1}^{r}\lambda_j\right)\frac{2}{M^4} | f (0) |^2 \int_{B (0, \delta / 2)} e^{- \varphi -
    \psi} d V,
  \end{eqnarray*}
  \[ \int_{B (0, \delta / 2)} | \partial^{\ast} (\beta^{\ast} \wedge d z)f |^2 e^{- \varphi
     - \psi} d V < \frac{3 r}{M^6} | f (0) |^2 \int_{B (0, \delta / 2)} e^{-
     \varphi - \psi} d V, \]
  and
  \[ \int_{B (0, \delta / 2)} | \beta^{\ast}|^2| f |^2 e^{- \varphi - \psi} d V <
     \frac{2 r}{M^4} | f (0) |^2 \int_{B (0, \delta / 2)} e^{- \varphi - \psi} d V.
  \]
  It then follows from \eqref{equ:key-estimate1} that
  \begin{eqnarray*}
    &  & \left(\sum_{j=1}^{r}\lambda_j\right)\frac{2}{M^4} | f (0) |^2 \int_{B (0, \delta / 2)} e^{- \varphi - \psi}
    d V + \frac{3 r}{M^6} | f (0) |^2 \int_{B (0, \delta / 2)} e^{- \varphi - \psi}
    d V\\
    & > & \int_{B (0, \delta / 2)} \left\langle \left[ i \partial \bar{\partial}
    \varphi \otimes \tmop{Id}_{\underline{\mathbb{C}}^{r + 1}}, \Lambda
    \right] \beta^{\ast} \wedge d z, \beta^{\ast} \wedge d z \right\rangle |f|^2 e^{- \varphi - \psi}
    d V\\
    &  & + \int_{B (0, \delta / 2)} | \partial^{\ast} (\beta^{\ast}\wedge d z )f |^2 e^{-
    \varphi - \psi} d V\\
    & > & - 2\varepsilon- k_{\varepsilon} \int_{B (0, \delta / 2)} | \beta^{\ast} |^2|f |^2 e^{-
    \varphi - \psi} d V \\
    & > & - 2\varepsilon- k_{\varepsilon} \frac{2 r}{M^4} | f (0) |^2 \int_{B (0, \delta / 2)}
    e^{- \varphi - \psi} d V ,
  \end{eqnarray*}
  equivalently,
  \begin{equation*}
    \frac{2\varepsilon M^4}{| f (0) |^2 \int_{B (0, \delta / 2)} e^{- \varphi -
    \psi} d V} + 2 (k_0 + \varepsilon) r + \frac{3 r}{M^2} > -2 \sum_{j
    = 1}^r \lambda_j>0.
  \end{equation*}
    We remark that $k_0 \rightarrow 0$ as $M \rightarrow \infty$. Indeed, it
  follows from \eqref{equ:2nd-form1} and \eqref{equ:2nd-form2} that if $M
  \rightarrow \infty$, $\beta^{\ast} \wedge \beta$ tends to zero, and by the
  definition of $k_0$, one has $k_0 \rightarrow 0$ as $M \rightarrow \infty$.
  We may choose $M$ large enough and $\varepsilon>0$ small enough such that
  \[
  2 (k_0 + \varepsilon) r + \frac{3 r}{M^2}<- \frac{1}{6} \sum_{j
    = 1}^r \lambda_j.
  \]
  For the fixed $M$ and $\varepsilon$ as above, we can choose $k$ large enough such that
  \[
    \frac{2\varepsilon M^4}{| f (0) |^2 \int_{B (0, \delta / 2)} e^{- \varphi -
    \psi} d V}<- \frac{1}{6} \sum_{j
    = 1}^r \lambda_j,
  \]
  since
  \[ \int_{B (0, \delta / 2)} e^{- \varphi - \psi} d V = \int_{B (0, \delta / 2)} e^{-
     \varphi - (k + k_{\varepsilon}) (| z |^2 - \delta^2 / 4)} d V \rightarrow
     \infty, \]
  as $k \rightarrow \infty$. Thus we have
  \[
    - 2 \sum_{j
    = 1}^r \lambda_j<\frac{2\varepsilon M^4}{| f (0) |^2 \int_{B (0, \delta / 2)} e^{- \varphi -
    \psi} d V} + 2 (k_0 + \varepsilon) r + \frac{3 r}{M^2}<- \frac{1}{3} \sum_{j
    = 1}^r \lambda_j,
  \]
this is a contradiction.
\end{proof}

\begin{remark}
In \cite{deng-ning-wang2021} and \cite{deng-ning-wang-zhou2023}, the property of optimal $L^p$ estimates
and the property of optimal $L^p$ extensions are introduced. It is shown that both the $L^2$ estimate and
$L^2$ extension properties characterize plurisubharmonic functions.
Combined with Theorem \ref{thm:2-division}, plurisubharmonic functions can be characterized by the optimal $L^2$ estimate property,
the optimal $L^2$ extension property and the existence of $f\in A^2(D,\varphi)$ which is $L^2$ divisible with respect to $\varphi$. In view of \cite{ohsawa2002} (see also \cite{demailly-smallbook}), certain types of $L^2$ division theorems can be deduced from the $L^2$ extension theorem under appropriate curvature conditions. It seems that the existence of $f\in A^2(D,\varphi)$ which is $L^2$ divisible with respect to $\varphi$ is a relatively weak condition for characterizing plurisubharmonicity.

  It also follows from Theorem \ref{thm:2-division} that if a bounded domain $D\subset\mathbb{C}^n$ admits a smooth exhaustion function $\varphi$ such that there exists $f\in A^2(D,\varphi)$ which is $L^2$ divisible with respect to $\varphi$, then
  $D$ is pseudoconvex. In particular, $D$ admits a complete K\"{a}hler metric.
\end{remark}

\subsection{An application to the complex Pr\'ekopa theorem}
\label{sec:application-complex-prekopa}

We apply Theorems \ref{thm:sharper-division} and \ref{thm:2-division}
to give a division proof of the following 
complex Pr\'ekopa theorem obtained by Berndtsson.

\begin{theorem}[Berndtsson \cite{berndtsson1998}]
\label{thm:application-complex-prekopa}
Let $U\subset\mathbb{C}^n$ be a domain and let
$V\subset\mathbb{C}^m$ be a pseudoconvex Reinhardt domain.
Suppose that $\Phi$ is plurisubharmonic on $U\times V$ and satisfies
\[
 \Phi(z,e^{i\theta_1}w_1,\ldots,e^{i\theta_m}w_m)=\Phi(z,w)
 \quad\text{for all }\theta\in\mathbb{R}^m.
\]
Then
\[
 \widetilde\Phi(z)=-\log\int_V e^{-\Phi(z,w)}\,dV_w
\]
is plurisubharmonic on $U$.
\end{theorem}

\begin{proof}
We first assume that $V$ is bounded and $\Phi$ is a real $C^2$ function
on a neighborhood of $U\times\overline V$.
Fix a ball $B\Subset U$. For $\varepsilon>0$, set
\[
 V_\varepsilon(z)=\int_V
 e^{-\Phi(z,w)-\varepsilon|w|^2}\,dV_w,
 \qquad \varphi_\varepsilon=-\log V_\varepsilon.
\]
The regularity assumption and compactness of $\overline B\times\overline V$
give $\varphi_\varepsilon\in C^2(B)$ and
$1\in A^2(B,\varphi_\varepsilon)$.

Let $(\psi,g)$ be any pair satisfying the conditions in the definition
of $L^2$ divisibility for $1$ and the weight
$\varphi_\varepsilon$ on $B$. Write
$g=(g_1,\ldots,g_k)\in\mathcal O(B)^{\oplus k}$, where $k$
is the number of functions in the division tuple.
On  $B\times V$, put
\[
 \widetilde g(z,w)=g(z),\qquad
 \Psi_\varepsilon(z,w)=\psi(z)+\varepsilon|w|^2.
\]
Let $\beta$ and $\widetilde\beta$ be the second fundamental forms
associated with $g$ and $\widetilde g$, respectively. Denote by
$\Lambda$ and $\widetilde\Lambda$ the adjoints of wedging with the
Euclidean K\"ahler forms on $B$ and $B\times V$, respectively.
If we write
$\beta=\sum_{j=1}^n\beta_j(z)\,dz_j$, then
$\widetilde\beta=\sum_{j=1}^n\beta_j(z)\,dz_j$ since $\widetilde g$ is independent of $w$. Therefore
\[
\left[
 i\partial\bar{\partial}\Psi_\varepsilon
   \otimes\tmop{Id}_{\underline{\mathbb{C}}^{k}}
   +i\widetilde\beta^{\ast}\wedge\widetilde\beta,
 \widetilde\Lambda
 \right]=\left[
 i\partial\bar{\partial}\psi
   \otimes\tmop{Id}_{\underline{\mathbb{C}}^{k}}
   +i\beta^{\ast}\wedge\beta,\Lambda
 \right]\oplus\varepsilon\tmop{Id}>0.
\]
Put $dz=dz_1\wedge\cdots\wedge dz_n$ and
$dZ=dz\wedge dw_1\wedge\cdots\wedge dw_m$.
Since $\widetilde\beta^{\ast}$ has no $d\bar w_\nu$ terms,
the preceding decomposition and Fubini's theorem give
\begin{equation}\label{eq:application-prekopa-I}
\begin{aligned}
 &I^{B\times V}_{\Phi,\Psi_\varepsilon,\widetilde g}(1)\\
=&\int_{B\times V}\Biggl\{\frac{1}{|\widetilde g|^2}
 +\Biggl(
 \left[i\partial\bar\partial\Psi_\varepsilon
 \otimes\tmop{Id}_{\underline{\mathbb C}^{k}}
 +i\widetilde\beta^{\ast}\wedge\widetilde\beta,
 \widetilde\Lambda\right]^{-1}
 \widetilde\beta^{\ast}\wedge dZ,
 \widetilde\beta^{\ast}\wedge dZ\Biggr)\Biggr\}e^{-\Phi-\Psi_\varepsilon}\,dV_zdV_w\\
=&\int_B\Biggl\{\frac{1}{|g|^2}
 +\Biggl(
 \left[i\partial\bar\partial\psi
 \otimes\tmop{Id}_{\underline{\mathbb C}^{k}}
 +i\beta^{\ast}\wedge\beta,\Lambda\right]^{-1}
 \beta^{\ast}\wedge dz,
 \beta^{\ast}\wedge dz\Biggr)\Biggr\}\\
&\qquad\cdot e^{-\psi(z)}
 \left(\int_Ve^{-\Phi(z,w)-\varepsilon|w|^2}\,dV_w\right)dV_z\\
=&I^B_{\varphi_\varepsilon,\psi,g}(1).
\end{aligned}
\end{equation}
Theorem \ref{thm:sharper-division} therefore supplies
$H\in\mathcal O(B\times V)^{\oplus k}$ such that
\[
 g\cdot H=1,\qquad
 \int_{B\times V}|H|^2e^{-\Phi-\psi-\varepsilon|w|^2}\,dV_zdV_w
 \leqslant I^B_{\varphi_\varepsilon,\psi,g}(1).
\]

Let $d\mu$ be normalized Haar measure on $\mathbb T^m$ and average
$H$ over rotations of the $w$-coordinates:
\[
 (PH)(z,w)=\int_{\mathbb T^m}H(z,e^{i\theta}w)\,d\mu(\theta).
\]
Note that $PH$ is holomorphic and
independent of $w$. We can write it as $h(z)$.
Averaging the division identity yields $g\cdot h=1$.
By the Cauchy--Schwarz inequality and rotation invariance of the measure and weight,
\begin{align*}
 \int_B|h|^2e^{-\varphi_\varepsilon-\psi}\,dV_z
 &=\int_{B\times V}|PH|^2
             e^{-\Phi-\psi-\varepsilon|w|^2}\,dV_zdV_w\\
 &\leqslant\int_{B\times V}|H|^2
             e^{-\Phi-\psi-\varepsilon|w|^2}\,dV_zdV_w\\
 &\leqslant I^B_{\varphi_\varepsilon,\psi,g}(1).
\end{align*}
Thus $1$ is $L^2$ divisible with respect to $\varphi_\varepsilon$ on
$B$. Theorem \ref{thm:2-division} shows that $\varphi_\varepsilon$
is plurisubharmonic. Since $\varphi_{\varepsilon}$ increasing to the continuous function $\tilde{\Phi}$ as $\varepsilon\to 0$,  $\widetilde\Phi$ is plurisubharmonic
on $B$, and therefore on $U$. This proves the assertion under the
additional boundedness and regularity assumptions.

For the general case, assume $\Phi\not\equiv-\infty$ and fix
$B\Subset U$. Choose an exhaustion $V_j\Subset V_{j+1}\Subset V$
by bounded pseudoconvex Reinhardt domains, using a torus-averaged
strictly plurisubharmonic exhaustion of $V$.
For each $j$, convolution with smooth radial kernels gives smooth
rotation-invariant plurisubharmonic functions
$\Phi_{j,\ell}\downarrow\Phi$ on a neighborhood of
$\overline B\times\overline V_j$ \cite{demailly-bigbook}.
The case already proved shows that
\[
 \widetilde\Phi_{j,\ell}(z)
 =-\log\int_{V_j}e^{-\Phi_{j,\ell}(z,w)}\,dV_w
\]
is plurisubharmonic on $B$.
By monotone convergence,
\[
 \begin{aligned}
 \widetilde\Phi_{j,\ell}
 &\downarrow\widetilde\Phi_j
 :=-\log\int_{V_j}e^{-\Phi(z,w)}\,dV_w
 &&(\ell\to\infty),\\
 \widetilde\Phi_j
 &\downarrow-\log\int_Ve^{-\Phi(z,w)}\,dV_w
 =\widetilde\Phi
 &&(j\to\infty).
 \end{aligned}
\]
The decreasing-limit property of plurisubharmonic functions proves
the conclusion on $B$, and hence on $U$. The case
$\Phi\equiv-\infty$ is immediate.
\end{proof}
\begin{remark}
  Similar methods can be used to prove the complex Pr\'{e}kopa theorem in a more general setting (e.g., allowing for more general compact Lie group actions on the fibers, cf. \cite{deng-hu-jiang2023}), with the help of Theorem \ref{thm:2-division}.
\end{remark}

\end{sloppypar}
\end{document}